\documentclass{amsart}
\usepackage{hyperref}
\usepackage{mathtools}
\usepackage{color}
\usepackage[margin=1in]{geometry}
\usepackage{amsmath,amsfonts,amsthm,amssymb}
\usepackage{setspace}
\usepackage{Tabbing}
\usepackage{fancyhdr}
\usepackage{lastpage}
\usepackage{extramarks}
\usepackage{chngpage}
\usepackage{mathrsfs}
\usepackage{esint}
\usepackage{xcolor}
\usepackage{enumitem}
\usepackage{comment}

\newtheorem{theorem}{Theorem}[section]
\newtheorem{proposition}[theorem]{Proposition}
\newtheorem{lemma}[theorem]{Lemma}

\newtheorem{conjecture}[theorem]{Conjecture}

\DeclareMathOperator{\essinf}{ess\, inf}

\NewDocumentCommand{\ind}{}{\chi}

\newcommand{\Z}{\mathbb{Z}}

\newcommand{\R}{\mathbb{R}}

\DeclareMathOperator{\esssup}{ess\,sup}

\NewDocumentCommand{\dmu}{}{\ \mathrm{d}\mu}

\NewDocumentCommand{\dy}{}{\ \mathrm{d}y}

\begin{document}
\title[Spectrality of Product Sets with a Perturbed Interval Factor]{Spectrality of Product Sets with a Perturbed Interval Factor}

\author[Ramabadran and van Vliet]{Aditya Ramabadran and Johannes van Vliet}

\address{Aditya Ramabadran
\newline \indent Department of Mathematics, University of California, Berkeley\indent 
\newline \indent  Evans Hall, Berkeley, CA 94720 \indent }
\email{aditya.ramabadran@berkeley.edu}

\address{Johannes van Vliet
\newline \indent Department of Mathematics, University of Washington\indent 
\newline \indent  C138 Padelford Hall Box 354350, Seattle, WA 98195,\indent }
\email{johov@uw.edu}

\date{\today}

\begin{abstract}
A set $\Omega \subset \mathbb{R}^d$ is said to be spectral if $L^2(\Omega)$ admits an orthogonal basis of exponentials. While the product of spectral sets is known to be spectral, the converse fails in general. In this paper, we prove that the converse holds when one factor is a perturbation of an interval: if $E \subset [0,3/2 - \epsilon]$ and $F$ are bounded sets of measure $1$, then $E \times F$ is spectral if and only if both $E$ and $F$ are spectral.
\end{abstract}
\maketitle

\section{Introduction}

Suppose $\Omega \subset \R^d$ is a bounded set of positive measure. If for some countable set $\Lambda \subset \R^d$, $\{ \exp(2\pi i \langle\lambda,x\rangle ) \}_{\lambda \in \Lambda}$ is an orthonormal basis of $L^2(\Omega)$, we say that $\Omega$ is \textit{spectral} and $\Lambda$ is a \textit{spectrum} for $\Omega$. Most notably, Fuglede conjectured that a set $\Omega$ is spectral if and only if it can tile the space $\R^d$ by translations \cite{Fug74}. Although counterexamples were found in dimensions $d \ge 3$ \cite{T05, Mat05, KM06}, Fuglede's conjecture sparked many interesting results, and has been proven in many special cases, such as when $\Omega$ is a convex body \cite{LM22}.

One interesting area of research in spectrality is in product domains. If $E \subset \R^n, F \subset \R^m$ are bounded sets of positive measure, then $E, F$ being spectral in $\R^n$ and $\R^m$ implies that $\Omega = E \times F$ is spectral in $\R^n \times \R^m$, since the product of spectra for $E$ and $F$ forms a spectrum for $\Omega$ \cite{Kol24}. In \cite{K16}, Kolountzakis conjectured that the converse is true as well: 

\begin{conjecture}
    Let $E \subset \R^n, F \subset \R^m$ be bounded and measurable sets. Then $\Omega = E \times F$ is spectral if and only if $E$ is spectral in $\R^n$ and $F$ is spectral in $\R^m$. 
\end{conjecture} 

If some spectrum $\Lambda$ for $\Omega$ has a product structure, i.e. $\Lambda = U \times V$, then $U$ and $V$ can serve as spectra for the factors $A$ and $B$, but in general we cannot assume that the spectrum $\Lambda$ has such a product structure. Nevertheless, there have been results in certain special cases. 

In \cite{GL16}, Greenfeld and Lev proved that this conjecture holds when one of the factors, i.e. $E$, is an interval in $\R$. \cite{K16} proved that it holds when $E$ is the union of two intervals in $\R$. \cite{GL20} then proved the conjecture when $E$ is a convex polygon in $\R^2$. \cite{KLM} proved that if $\Omega = E \times F$ where $E$ is a convex body in $\R^n$ and $F$ is any bounded, measurable set in $\R^m$, then $\Omega = E \times F$ being spectral implies that $E$ must be spectral (but there is no claim on the spectrality of $F$). Recently, \cite{Som24} proved that the conjecture is false in general by constructing a counterexample, a set $\Omega = E \times F \subset \R^6$ such that $E \subset \R^3$ is not spectral. 

In this paper, we focus on spectral product sets where one of the factors is a perturbation of an interval. This is inspired by \cite{KL02} which focuses on sets $E \subset \R$ contained in intervals of length strictly less than $3|E|/2$. Specifically, they prove the following: 

\begin{theorem}
Suppose $E \subseteq [0,L]$ is measurable with measure 1 and $L = 3/2 - \epsilon$ for some $\epsilon > 0$. Let $\Lambda \subset \mathbb{R}$ be a discrete set containing 0. Then
\begin{enumerate}
    \item[(a)] if $E+\Lambda = \mathbb{R}$ is a tiling, it follows that $\Lambda = \mathbb{Z}$.
    \item[(b)] if $\Lambda$ is a spectrum of $E$, it follows that $\Lambda = \mathbb{Z}$.
\end{enumerate}
\label{thm_KL}
\end{theorem} 

They note that this upper bound $L < 3/2$ is optimal since the set $[0,1/2) \cup [1,3/2]$ is contained in $[0,3/2]$, tiles $\mathbb{R}$ with translation set $\{0,1/2\} + 2\mathbb{Z}$, and has the spectrum $\{0,1/2\} + 2\mathbb{Z}$, so it has neither a lattice translation set nor a lattice spectrum. 

Our main theorem is the following:

\begin{theorem}
Let $E$ be measurable with measure $1$ and $E\subset[0,3/2-\epsilon]$ for some $\epsilon > 0$, and let $F$ be bounded and measurable with measure $1$. If $E \times F$ is spectral, then both $E$ and $F$ are spectral.
\label{thm_main}
\end{theorem} 

To show that $E$ is spectral, we use a ``weak tiling" condition, which was first applied to Fuglede's conjecture in \cite{LM22}. A theorem proven in \cite{K16} tells us that if $F$ were not spectral, then $\hat{\chi_E}$ would admit a certain sequence of roots. We then derive a contradiction using Jensen's formula.

\subsection*{Acknowledgement}

We would like to thank our mentor, Bobby Wilson, for his support. The authors both participated in an REU led by Professor Wilson in the summer of 2024, and he continued to provide valuable advice on this project throughout the year. The REU was funded by NSF CAREER Fellowship, DMS 2142064.


\section{Preliminaries and Results from the Literature}

We will begin by discussing five important results in the literature that we will need for our proof. 

The first result we need is the following lemma, proven in \cite{KL02}. This lemma is crucial for their proof of Theorem \ref{thm_KL}, and the proof of the lemma simply relies on taking cases on if $t$ is in $[0, \frac12], (\frac12, \frac34], $ or $(\frac34, 1)$ and arguing by contradiction that there isn't enough ``space" if there is no overlap.

\begin{lemma}
    Suppose that $E\subseteq[0,L]$ is measurable with measure $1$ and that $L=3/2-\epsilon$ for some $\epsilon>0$. Then $m(E\cap(E+t))>0$ whenever $0\leq t<1$.
\end{lemma}

This next result is proven in  \cite{K99}. We will eventually be able to show that, for some discrete set $\Lambda$ containing zero, $E + \Lambda$ is a tiling and $[0,1] + \Lambda$ is a packing. This theorem will allow us to say that $[0,1] + \Lambda$ is then a tiling, which forces $\Lambda = \Z$ to be a tiling set for $E$. 

\begin{theorem}
    If $f,g\ge 0$, $\int f(x)dx=\int g(x)dx=1$, and both $f+\Lambda$ and $g+\Lambda$ are packings of $\mathbb{R}^d$, then $f+\Lambda$ is a tiling if and only if $g+\Lambda$ is a tiling.
\end{theorem}

The third result below is proven in \cite{K16} by Kolountzakis, though a simpler proof of a stronger version of this theorem can be found in Section 4 of \cite{GL20}. It will be key for us to show that the second factor $F$ in $E \times F$ is spectral, after showing that $E$ is spectral. Kolountzakis uses it solely with $D = (-1/2,1/2)$ in \cite{K16}, but the theorem is general enough to allow us to use it with varying sets $D_n$. By supposing for contradiction that $F$ is not spectral and translating $D_n$, we are able to produce a sequence of zeros of $\hat{\chi_E}$. We will use this fact to force a contradiction. 

\begin{theorem}
    Suppose $\Omega=E\times F\subset\mathbb{R}^m\times\mathbb{R}^n$ has $m(E)=m(F)=1$, and suppose also that the bounded set $D\subset \mathbb{R}^m$ is such that
    $$(D-D)\cap\{\hat{\chi_E}=0\}=\varnothing.$$

    \begin{enumerate}
        \item If $m(D)=1$ and $\Omega$ is spectral then $F$ is also spectral.
        \item If $m(D)>1$ then neither $E$ nor $\Omega$ is spectral.
    \end{enumerate}

\end{theorem}

The fourth result Theorem 2.4 is proven in \cite{LM22} and is an important tool in their proof of Fuglede's conjecture for convex bodies. They show that spectrality of $\Omega$ implies that $\Omega$ ``weakly tiles" its complement $\Omega^c$. This link between spectrality and a notion of tiling has been useful far beyond their original paper. The fifth result below it, Theorem 2.5, is proven in \cite{KLM}. This result establishes that Conjecture 1.1 is true if the condition of spectrality is replaced by the weak tiling condition of Theorem 2.4. We will use both of these results in conjunction to show that $E$ and $F$ both weakly tile their complements, assuming that $E \times F$ is spectral. 

\begin{theorem}
    Let $\Omega$ be a bounded, measurable set in $\mathbb{R}^d$. If $\Omega$ is spectral, then its complement $\Omega^c=\mathbb{R}^d\smallsetminus\Omega$ admits a weak tiling by translates of $\Omega$. That is, there exists a positive, locally finite measure $\mu$ such that $\chi_\Omega*\mu=\chi_{\Omega^c}$ a.e.
\end{theorem}

\begin{theorem}
    Let $E\subset \mathbb{R}^n$ and $F\subset\mathbb{R}^m$ be two bounded, measurable sets. Then the product set $E\times F$ can weakly tile its complement in $\mathbb{R}^n\times\mathbb{R}^m$ by translations if and only if both $E$ and $F$ weakly tile their complements in $\mathbb{R}^n$ and $\mathbb{R}^m$, respectively.
\end{theorem}


\section{Main Arguments}\label{s:main}

\subsection{Perturbation of interval admits lattice tiling} 

\begin{proposition}
Let $m(E)=1$, $\sup E=\esssup E$, $\inf E=\essinf E$, and $E\subset [0,L]$, where $L=3/2-\epsilon$ for some $\epsilon>0$. Suppose $E$ weakly tiles its complement with respect to a locally finite measure $\mu$. Then $\mu=\sum_{k\in\mathbb{Z}\smallsetminus \{0\}}\delta_k$.
\end{proposition}

\begin{proof}
    We will first show that $\mu(-1,1)=0$. Let $K(t) = (\ind_E \ast \ind_{-E})(t) = \int \ind_E(y) \ind_{E}(y-t) \dy = \int \ind_{E \cap (E+t)}(y) \dy = m(E \cap (E+t))$. Then since $\ind_E, \ind_{-E}$ are in $L^2$, $K$ is uniformly continuous. $K$ is also bounded since $m(E \cap (E+t)) \le 1$, and it is compactly supported since $E \subset [0, \frac 32 - \epsilon]$ so for $|t| > \frac 3 2$, $K(t) = 0$. By Lemma 2.1, $K(t) > 0$ for all $t \in [0,1)$. And since Lebesgue measure is translation invariant, $K$ is even and thus $K>0$ on $(-1,1)$. By the Extreme Value Theorem, for each $n$ there exists $m_n>0$ such that $K(t) \ge m_n$ whenever $t \in [-1+\frac{1}{n}, 1 - \frac 1 n]$. 
    
    $E$ weakly tiles its complement with respect to $\mu$, so $\ind_E \ast \mu = \ind_{E^c}$ $m$-almost everywhere.   This means that $\ind_E \ast \mu \ast \ind_{-E} = \ind_{E^c} \ast \ind_{-E}$ a.e. But also, $\ind_E \ast \mu \ast \ind_{-E} = K \ast \mu$ a.e. So $\mu \ast K = \ind_{E^c} \ast \ind_{-E} = (1 - \ind_E) \ast \ind_{-E}$. Then $(1 \ast \ind_{-E})(t) = \int \ind_E(y-t) \dy =m(E)= 1$, so we get $\mu \ast K = 1 - K$ a.e. Since $K$ is uniformly continuous with compact support and $\mu$ is locally finite, the estimate $|(\mu \ast K)(x) - (\mu \ast K)(y)| \le \int |K(x-z) - K(y-z)| \dmu(z)$ tells us that $\mu*K$ is continuous. Since both $1-K$ and $\mu*K$ are continuous and are equal to one another a.e., they are equal everywhere. In particular, we have $(\mu \ast K)(0) = 1 - K(0) = 0$. Then for each $n$, $$0 = \int K(-y) \dmu(y) \ge \int_{[-1 + \frac 1 n, 1-\frac{1}{n}]} K(-y) \dmu(y) \ge m_n \: \mu([-1+\frac 1 n, 1-\frac{1}{n}]),$$ so $\mu([-1 + \frac 1 n, 1-\frac{1}{n}]) = 0$ for all $n$. By continuity from below, we get $\mu((-1, 1)) = 0$. 
    
    Let $r=\sup E=\esssup E$ and let $t=\esssup\left([0,r-1]\cap E^c\right)$. Assume that $t<r-1$. Let $g(x)=\chi_E*\chi_{(1-r,-t)}(x)=m(E\cap((t,r-1)+x))$. Then $g$ is uniformly continuous and $g*\mu$ is continuous by the same reasoning which showed that $K$ and $K*\mu$ were continuous and uniformly continuous, respectively. Also, $g>0$ on $(1,r-t)$ by the minimality of $r$, so for each $n$ there exists $m_n$ such that $g\ge m_n>0$ on $[1+1/n,r-t-1/n]$ by the Extreme Value Theorem. Moreover, because $\chi_E*\mu=1-\chi_E$ almost everywhere, $g*\mu=(1-\chi_E)*\chi_{(1-r,-t)}=r-1-t-g$ almost everywhere. Since $g$ and $g*\mu$ are both continuous, we have equality everywhere, and in particular $g*\mu(0)=r-1-t-g(0)$. And by the maximality of $t$, $g(0)=m(E\cap(t,r-1))=r-t-1$, so $g*\mu(0)=0$. So for any $n$,
    $$0=g*\mu(0)=\int g(-y)d\mu(y)\ge \int_{[t-r+1/n,-1-1/n]}g(-y)d\mu(y)\ge m_n\mu[t-r+1/n,-1-1/n],$$
    which shows that $\mu(t-r,-1)=0$. Furthermore,
    $$\{x\in (t,r-1):-1\in x-E\}=\left((t+1,r)\cap E\right)-1,$$
    and since $m((t+1,r)\cap E)>0$, $\{x\in (t,r-1):-1\in x-E\}$ has positive measure. And since almost all $x\in (t-1,r)$ are contained in $E$ and almost all $x\in E$ satisfy $\mu(x-E)=0$, this shows that $\mu\{-1\}=0$. Thus $\mu(t-r,-1]=\mu(t-r,1)=0$.

    Assume now that $t\leq r-1$. It is still true, then, that $\mu(t-r,1)=0$. By the minimality of $t$, we know that $m(E^c\cap (t-\delta,t))>0$ for any $\delta>0$. This means that we may choose a sequence $y_n$ which increases to $t$ such that $\mu(y_n-E)=1$ for each $n$. Then since $\mu(t-r,1)=0$, we have
    $$\mu(y_n-E)=\mu\left((y_n-E)\cap(-\infty,t-r]\right)\leq \mu(y_n-r,t-r],$$
    and so $\mu\{t-r\}\ge 1$ by continuity from above since $\mu$ is locally finite. 
    
    We know that $t-r\leq -1$, and $m(E\cap[0,1])>0$, so $m(E^c\cap(E+t-r))>0$. Also, $\mu(x-E)=1$ for almost all $x\in E^c$, which means that for some $x\in E^c\cap(E+t-r)$, $\mu(x-E)=1$. Because $x\in E+t-r$, $t-r\in x-E$, which shows that $\mu\{t-r\}\leq\mu(x-E)\leq 1$. Thus $\mu\{t-r\}=1$.

    Let $\mu_1(A)=(\mu+\delta_0-\delta_{t-r})(A+t-r)$, so that $\mu_1$ is a locally finite Borel measure. We have
    $$\chi_E*\mu_1=\tau_{r-t}((\mu+\delta_0)*\chi_E)-\delta_{t-r}(x-E+t-r)=1-\chi_E=\chi_{E^c},$$
    where the third equality holds a.e. due to the fact that $\chi_E*\mu=\chi_{E^c}$ a.e. The reasoning above then shows that $\mu_1(t-r,0)=0$ and that $\mu_1\{t-r\}=1$. Therefore
    $$\mu(2(t-r),t-r)=(\mu+\delta_0-\delta_{t-r})(2(t-r),t-r)=\mu_1(t-r,0)=0,$$
    and
    $$\mu\{2(t-r)\}=(\mu+\delta_0-\delta_{t-r})\{2(t-r)\}=\mu_1\{t-r\}=1.$$
    Continuing in this manner, we find that $\mu|_{(-\infty,0]}=\sum_{j=1}^\infty \delta_{j(t-r)}$. 
    
    Let $E'=r-E$, and let $\mu'(A)=\mu(-A)$. Then $E'$ satisfies the hypotheses of the Proposition and $\mu'$ is a locally finite Borel measure. Also,
    $$\chi_{E'}*\mu'(x)=\mu(-x+r-E)=\chi_E*\mu(-x+r)=\chi_{(E')^c}(x),$$
    with the last equality holding a.e. due to the fact that $\chi_E*\mu=\chi_{E^c}$ a.e. Let $r'=\sup E'=\esssup E'$ and $t'=\esssup[0,r'-1]\cap (E')^c$. Then arguing as we have above, we find that $\mu|_{[0,\infty)}=\sum_{k=1}^\infty \delta_{k(r'-t')}$. This means that $\mu+\delta_0=\sum_{\lambda\in \Lambda}\delta_\lambda$, where $\Lambda$ is a discrete set and $|\lambda-\lambda'|\ge 1$ whenever $\lambda\neq\lambda'$. Therefore $[0,1]+\Lambda$ is a packing. And since $E+\Lambda$ is a tiling, $\Lambda$ must be a tiling set for $[0,1]$ by Theorem 2.2. Because $0\in\Lambda$, we must have $\Lambda=\mathbb{Z}$. It follows that $\mu=\sum_{k\in\mathbb{Z}\smallsetminus\{0\}}\delta_k$. \end{proof}


\subsection{Spectrality of second factor in product}

\begin{lemma}
    Let $m(E)=1$, $E\subset[0,3/2-\epsilon]$, and suppose $E$ tiles $\mathbb R$ with respect to $\mathbb{Z}$. Then $\hat{\chi_E}(\xi)\neq 0$ whenever $-1/2\leq\xi\leq 1/2$.
\end{lemma}

\begin{proof}
    Since $E$ tiles with respect to $\mathbb{Z}$ and $E\subset [0, 3/2-\epsilon]$, there is a set $F\subset [0,1/2-\epsilon]$ such that $\chi_E=\chi_{[0,1]}-\chi_F+\tau_1\chi_F$ a.e., where $\tau_hf(x):=f(x-h)$. We may assume without loss of generality that $m(F)> 0$, for otherwise $\chi_E=\chi_{[0,1]}$ almost everywhere, and so the zeros of $\hat{\chi_E}$ are precisely the non-zero integers. If $\xi\neq 0$,
    $$\hat{\chi_E}(\xi)=(e^{-2\pi i\xi}-1)\left(\hat{\chi_F}(\xi)-\frac{1}{2\pi i\xi}\right).$$
    The roots of $\hat{\chi_E}$, then, are the nonzero roots of $e^{-2\pi i\xi}-1$, together with the roots of $\hat{\chi_F}(\xi)-1/(2\pi i\xi)$. The roots of $e^{-2\pi i\xi}-1$ are precisely the integers. If $-1/2\leq \xi\leq 1/2$, then $|2\pi x\xi|<\pi/2$ for all $x\in F$, since $F\subset [0,1/2-\epsilon]$. This means that $\cos(2\pi x\xi)>0$ on $F$, so
    
    $$\Re(\hat{\chi_F}(\xi))=\int_F\cos(2\pi x\xi)dx=\int_F\left|\cos(2\pi x\xi)\right|dx.$$
    If the integral on the right-hand side were zero, then we would have $\cos(2\pi x\xi)=0$ a.e. on $F$. But this is a contradiction, as $m(F)>0$ by assumption. Therefore $\hat{\chi_F}(\xi)-1/(2\pi i\xi)\neq 0$, which is what we needed to show. \end{proof}

    \begin{proposition}
    Let $m(E)=1$, $E\subset[0,3/2-\epsilon]$, and suppose $E$ tiles $\mathbb{R}$ with respect to $\mathbb{Z}$. Assume $m(F)=1$. If $E\times F$ is spectral, then $F$ is spectral.
\end{proposition}

\begin{proof}
    Suppose $E\subset[0,3/2-\epsilon]$, $m(E)=1$, and that $E$ tiles $\mathbb{R}$ with respect to $\mathbb{Z}$. Then $E$ admits $\mathbb{Z}$ as a spectrum as proven in \cite{Fug74}. If $\hat{\chi_E}(\xi)\neq 0$ for $\xi\in (-1,1)$, then we could apply Theorem 2.3 with $D=(0,1)$ as above and conclude that $F$ is spectral\footnote{This is not always the case though; for example, with $E = [0, \frac12] \cup [\frac34, \frac54]$, $\hat{\chi_E}$ has a zero at $\frac23 \in (-1,1)$.}. Therefore, by Lemma 3.2, we have that $t_0:=\inf\{\xi\in (0,1):\hat{\chi_E}(\xi)=0\}\in (1/2,1)$; note that $\hat{\chi_E}(t_0)=0$ by the Identity Theorem. 
    
    Suppose for contradiction that $F$ is not spectral. For each domain $D_n:=(0,t_0)\cup(n-1+t_0,n)$, we have a zero of $\hat{\chi_E}$ in $D_n-D_n$ by Theorem 2.3 once again. The difference between two points in $(n-1+t_0,n)$ must lie in $(-t_0,t_0)$, since $1-t_0<t_0$. The difference between a point of $(n-1+t_0,n)$ and a point in $(0,t_0)$ must either lie in $(n-1,n)$ or in $(-n,-n+1)$. Therefore $D_n-D_n=(-n,-n+1)\cup(-t_0,t_0)\cup (n-1,n)$. By the minimality of $t_0$, there are no zeros of $\hat{\chi_E}$ in $(-t_0,t_0)$. So for each interval $(n-1,n)$, $n\in\mathbb{Z}_+$, there is a zero, $b_n$, of $\hat{\chi_E}$ by Theorem 2.3 (1) and the fact that the zero sets of Fourier transforms of real-valued functions are symmetric about zero. This also means that $-b_n$ is a zero of $\hat{\chi_E}$ for each $n\in\mathbb{Z}_+$.

For $\rho>0$, since $\hat{\chi_E}(0)=1$ and $\hat{\chi_E}$ is an entire function, Jensen's formula tells us that
$$\sum_{j=1}^k\log\left(\frac{\rho}{|a_j|}\right)=\frac{1}{2\pi}\int_0^{2\pi}\log|\hat{\chi_E}(\rho e^{i\theta})|d\theta,$$
where $a_1,\cdots,a_k$ are the zeros of $\hat{\chi_E}$ in $D(0,\rho)\subset\mathbb{C}$. To estimate the integral on the right hand side, observe:
$$|\hat{\chi_E}(\rho e^{i\theta})|\leq\int_E|e^{-2\pi i x\rho(\cos(\theta)+i\sin(\theta))}|dx=\int_E e^{2\pi x\rho\sin(\theta)}dx.$$
For $\theta\in (\pi,2\pi)$, $\sin(\theta)<0$, so the right-hand side of the inequality above is bounded above by $1$. Thus $\log|\hat{\chi_E}(\rho e^{i\theta})|\leq 0$ for $\theta\in (\pi,2\pi)$. By contrast, $\sin(\theta)> 0$ 
whenever $\theta\in (0,\pi)$, so the map $x\mapsto e^{2\pi x\rho\sin(\theta)}$ is increasing in $x$. And since $E\subset [0,3/2-\epsilon]$, $x\leq 3/2$. Therefore
$$\int_E e^{2\pi  x\rho\sin(\theta)}dx\leq\int_E e^{3\pi\rho\sin(\theta)}dx=e^{3\pi\rho\sin(\theta)},$$
so long as $\theta\in (0,\pi)$. This gives us the estimate
$$\frac{1}{2\pi}\int_0^{2\pi}\log|\hat{\chi_E}(\rho e^{i\theta})|d\theta\leq \frac{1}{2\pi}\int_0^\pi 3\pi\rho\sin(\theta)d\theta=3\rho.$$
Suppose $\rho>N$ for some $N\in\mathbb{Z}_+$. Then the zeros $\{a_j\}$ of $\hat{\chi_E}$ in $D(0,\rho)$ include the nonzero integers in $[-N,N]$, as well as each $b_n$ and each $-b_n$ for $n\in [1,N]$. Therefore, since $|b_n|<n$ for each $n$,
$$\sum_{j=1}^k\log\left(\frac{\rho}{|a_j|}\right)\ge 2\sum_{n=1}^N\left(\log\frac{N}{n}+\log\frac{N}{b_n}\right)\ge 4\sum_{n=1}^N\log\frac{N}{n}=4(N\log N-\log (N!)).$$
Letting $\rho\to N$ shows that $4(N\log N-\log(N!))\leq 3N$ for each $N$. By Stirling's approximation (\cite{Ahlfors1978}), $\log N!=N\log N-N+O(\log(N))$. This means that $4(N\log N-\log(N!))=4N+O(\log N)$, which gives us a contradiction as $N\to\infty$. \end{proof}

\section{Conclusion}\label{s:argument}

We now deduce Theorem 1.3 from the results we have already proven.

\noindent\emph{Proof of Theorem 1.3.}
    Let $m(E)=1$, $E\subset [0, 3/2-\epsilon]$, let $F$ be bounded and measurable, and let $E\times F$ be spectral. By modifying $E$ (and thus $E\times F$) on a set of measure $0$, we may assume that $\sup E=\esssup E$ and that $\inf E=\essinf E$ without affecting the spectrality or tiling properties of $E$, $F$, or $E\times F$. By Theorem 2.4 and Theorem 2.5, $E$ weakly tiles its complement in $\mathbb{R}$. By Proposition 3.1, $E+\mathbb{Z}$ is a tiling. It is well known that this implies that $\mathbb{Z}$ is a spectrum for $E$, as proven in Fuglede's original paper \cite{Fug74}. By Proposition 3.3, $F$ is spectral. \qed


\bibliographystyle{plain}
\bibliography{references}









\end{document}